\documentclass[12pt]{article}

\usepackage{amsmath,enumerate,amsfonts,amssymb,color,graphicx,amsthm}

\def\RR{{\mathbb R}}

\def\diag{{\rm diag}}
\def\Snn{{\mathcal{S}^{n \times n}}}
\def\E{{\hat{\varepsilon}}}

\newcounter{marnote}

\begin{document}
\newtheorem{thm}{Theorem}[section]
\newtheorem{Def}[thm]{Definition}
\newtheorem{lem}[thm]{Lemma}
\newtheorem{rem}[thm]{Remark}
\newtheorem{question}[thm]{Question}
\newtheorem{prop}[thm]{Proposition}
\newtheorem{cor}[thm]{Corollary}
\newtheorem{example}[thm]{Example}

\title{Strong Comparison Principles for Some
 Nonlinear Degenerate Elliptic Equations}

\author{YanYan Li \footnote{Department of Mathematics, Rutgers University, 110 Frelinghuysen Rd, Piscataway, NJ 08854, USA. Email: yyli@math.rutgers.edu.}~\quad Bo Wang \footnote{Corresponding author. School of Mathematics and Statistics, Beijing Institute of Technology, Beijing 100081, China. Email: wangbo89630@bit.edu.cn.}}
\date{Dedicated to the memory of Xiaqi Ding}

\maketitle

\begin{abstract}
In this paper, we obtain the strong comparison principle and Hopf Lemma for locally Lipschitz viscosity solutions to a class of nonlinear degenerate elliptic operators of the form $\nabla^2 \psi + L(x,\nabla \psi)$, including the conformal hessian operator.



\end{abstract}

\setcounter{section}{0}

\section{Introduction}

In this paper, we establish the strong comparison principle and Hopf Lemma for locally Lipschitz viscosity solutions to a class of nonlinear degenerate elliptic operators.

For a positive integer $n\geq2$, let $\Omega$ be an open connected bounded subset of $\RR^{n}$, the $n$-dimensional euclidean space. For any $C^{2}$ function $u$ in $\Omega$, we consider a symmetric matrix function
\begin{equation}
F[u]:=\nabla^{2}u+L(\cdot,\nabla u),\label{Fform}
\end{equation}
where $L\in C^{0,1}_{loc}(\Omega\times\RR^{n})$, is in $\mathcal{S}^{n\times n}$, the set of all $n\times n$ real symmetric matrices.

One such matrix operator is the conformal hessian operator (see e.g. \cite{LiLi03}, \cite{V} and the references therein), that is,
\begin{equation*}
A[u]=\nabla^{2}u+\nabla u\otimes \nabla u-\frac{1}{2}|\nabla u|^{2}I,
\end{equation*}
where $I$ denotes the $n\times n$ identity matrix, and for $p$, $q\in\RR^{n}$, $p\otimes q$ denotes the $n\times n$ matrix with entries $(p\otimes q)_{ij}=p_{i}q_{j}$, $i$, $j=1$, $\cdots$, $n$. Some comparison principles for this matrix operator have been studied in \cite{Li07}-\cite{LiNir-misc}. Comparison principles for other classes of (degenerate) elliptic operators are available in the literature. See \cite{AmendolaGaliseVitolo13-DIE}-\cite{BirindelliDemengel07-CPAA}, \cite{CafLiNir11}-\cite{KoikeLey11-JMAA}, \cite{Trudinger88-RMI} and the references therein. 

Let $U$ be an open subset of $\mathcal{S}^{n\times n}$, satisfying
\begin{equation}
0\in\partial U,\quad U+\mathcal{P}\subset U,\quad tU\subset U,~\forall~t>0,\label{23}
\end{equation}
where $\mathcal{P}$ is the set of all non-negative matrices. Furthermore, in order to conclude that the strong comparison principle holds, we assume Condition $U_{\nu}$, as introduced in \cite{LiNir-misc}, for some unit vector $\nu$ in $\RR^{n}$: 
there exists $\mu=\mu(\nu)>0$ such that 
\begin{equation}
U+C_{\mu}(\nu)\subset U.\label{halfcone}
\end{equation}
Here $C_{\mu}(\nu):=\{t(\nu\otimes\nu+A):A\in\mathcal{S}^{n\times n},\|A\|<\mu,t>0\}$. Some counter examples for the strong maximum principle were given in \cite{LiNir-misc} to show that the condition (\ref{halfcone}) cannot be simply dropped.

\begin{rem}
If $U$ satisfies (\ref{23}),
\begin{equation*}
\diag\{1,0,\cdots,0\}\in U,
\end{equation*}
and 
\begin{equation*}
O^{t}UO\subset U,\quad\forall~O\in O(n),
\end{equation*}
where $O(n)$ denotes the set of $n\times n$ orthogonal matrices, then it is easy to see that $U$ satisfies (\ref{halfcone}).

\end{rem}

Let $u$, $v\in C^{0,1}_{loc}(\Omega)$. We say that 

\begin{equation}
F[u]\in\Snn\setminus U\quad \left(F[v] \in \bar{U}\right),\quad\mbox{in }\Omega \label{eq0}
\end{equation}
in the viscosity sense, if for any $x_{0}\in\Omega$, $\varphi\in C^{2}(\Omega)$, $(\varphi-u)(x_{0})=0$ ($(\varphi-v)(x_{0})=0$) and 
\begin{equation*}
u-\varphi\geq0\quad(v-\varphi\leq0),\quad\mbox{near }x_{0},
\end{equation*}
there holds
\begin{equation*}
F[\varphi](x_{0})\in\Snn\setminus U \quad \left(F[\varphi](x_{0})\in \bar{U}\right).
\end{equation*}

We have the following strong comparison principle and Hopf Lemma.

\begin{thm}(Strong Comparison Principle)
Let $\Omega$ be an open connected subset of $\RR^{n}$, $n\geq2$, $U$ be an open subset of $\mathcal{S}^{n\times n}$, satisfying (\ref{23}) and Condition $U_{\nu}$ for every unit vector $\nu$ in $\RR^{n}$, and $F$ be of the form (\ref{Fform}) with $L\in C_{loc}^{0,1}(\Omega\times\RR^{n})$. Assume that $u$, $v\in C^{0,1}_{loc}(\Omega)$ satisfy (\ref{eq0}) in the viscosity sense, $u\geq v$ in $\Omega$. Then either $u>v$ in $\Omega$ or $u\equiv v$ in $\Omega$.
\label{strong compare}
\end{thm}

\begin{thm}(Hopf Lemma)
Let $\Omega$ be an open connected subset of $\RR^{n}$, $n\geq2$, $\partial\Omega$ be $C^{2}$ near a point $\hat{x}\in\partial\Omega$, and $U$ be an open subset of $\mathcal{S}^{n\times n}$, satisfying (\ref{23}) and Condition $U_{\nu}$ for $\nu=\nu(\hat{x})$, the interior unit normal of $\partial\Omega$ at $\hat{x}$, and $F$ be of the form (\ref{Fform}) with $L\in C_{loc}^{0,1}(\Omega\times\RR^{n})$. Assume that $u$, $v\in C^{0,1}_{loc}(\Omega\cup\{\hat{x}\})$ satisfy (\ref{eq0}) in the viscosity sense, $u>v$ in $\Omega$ and $u(\hat{x})=v(\hat{x})$. Then we have 
\begin{equation*}
\liminf\limits_{s\rightarrow0^{+}}\frac{(u-v)(\hat{x}+s\nu(\hat{x}))}{s}>0.
\end{equation*}
\label{hopf}
\end{thm}

\begin{rem}
If $u$ and $v\in C^{2}$, then Theorems \ref{strong compare} and \ref{hopf} were proved in \cite{LiNir-misc}. \end{rem}

\section{Proof of Theorem \ref{strong compare}}

\begin{proof}[Proof of Theorem \ref{strong compare}]

We argue by contradiction. Suppose the conclusion is false. Since $u-v\in C^{0,1}_{loc}(\Omega)$ is non-negative, the set $\{x\in\Omega:u=v\}$ is closed. Then there exists an open ball $B(x_{0},R)\subset\subset\Omega$ centered at $x_{0}\in\Omega$ with radius $R>0$ such that 
\begin{equation*}
\begin{cases}
u-v>0,\quad\mbox{in }\overline{B(x_{0},R)}\texttt{\symbol{'134}}\{\hat{x}\},\\
u(\hat{x})-v(\hat{x})=0,\quad\hat{x}\in\partial B(x_{0},R).
\end{cases}
\end{equation*}

We make use of the standard comparison function
\begin{equation}
h(x):=e^{-\alpha|x-x_{0}|^{2}}-e^{-\alpha R^{2}},\quad\forall~\alpha>0,x\in\Omega.\label{comparef}
\end{equation}
For $i$, $j=1$, $\cdots$, $n$, we have 
\begin{equation}
h_{i}(x)=\frac{\partial }{\partial x_{i}}h(x)=-2\alpha(x_{i}-(x_{0})_{i})e^{-\alpha|x-x_{0}|^{2}},\label{hi}
\end{equation}
and 
\begin{equation}
h_{ij}(x)=\frac{\partial^{2}}{\partial x_{i}\partial x_{j}}h(x)=4\alpha^{2}e^{-\alpha|x-x_{0}|^{2}}\left[(x_{i}-(x_{0})_{i})(x_{j}-(x_{0})_{j})-\frac{1}{2\alpha}\delta_{ij}\right].
\label{hij}\end{equation}

Choose $0<R'<\frac{R}{2}$ such that $B(\hat{x},R')\subset\subset\Omega$. For any $\delta\in(0,R')$, we have that for any $x\in\overline{B(\hat{x},\delta)}$,
\begin{equation}
-1\leq h(x)\leq1,\quad|\nabla h(x)|+|\nabla^{2}h(x)|\leq C\label{319-}
\end{equation}
for some $C>0$ independent of $\delta$ and $\alpha$. 

It follows that, for any $0<\E<\min\limits_{(\overline{B(\hat{x},\delta)}\texttt{\symbol{'134}}B(\hat{x},\frac{1}{2}\delta))\cap\overline{B(x_{0},R)}}(u-v)$, 
\begin{equation}
u-v-\E h>0,\quad\mbox{on }\overline{B(\hat{x},\delta)}\texttt{\symbol{'134}}B(\hat{x},\frac{1}{2}\delta),\quad (u-v-\E h)(\hat{x})=0.\label{319}
\end{equation}
Indeed, by (\ref{319-}) and the fact that $h<0$ outside $\overline{B(x_{0},R)}$, for any $x\in(\overline{B(\hat{x},\delta)}\texttt{\symbol{'134}}B(\hat{x},\frac{1}{2}\delta))\texttt{\symbol{'134}}\overline{B(x_{0},R)}$,  
$$(u-v)(x)\geq0>\E h(x);$$
 and for any $x\in(\overline{B(\hat{x},\delta)}\texttt{\symbol{'134}}B(\hat{x},\frac{1}{2}\delta))\cap\overline{B(x_{0},R)}$, 
 $$(u-v)(x)\geq\min\limits_{(\overline{B(\hat{x},\delta)}\texttt{\symbol{'134}}B(\hat{x},\frac{1}{2}\delta))\cap\overline{B(x_{0},R)}}(u-v)>\E\geq\E h(x).$$

For any $\epsilon>0$, we define the $\epsilon$-lower and upper envelope of $u$ and $v$ as 
\begin{equation*}
u_{\epsilon}(x):=\min\limits_{y\in\overline{B(x_{0},R)}\cup\overline{B(\hat{x},R')}}\{u(y)+\frac{1}{\epsilon}|x-y|^{2}\},\quad\forall x\in\overline{B(x_{0},R)}\cup\overline{B(\hat{x},R')},
\end{equation*}
and 
\begin{equation*}
v^{\epsilon}(x):=\max\limits_{y\in\overline{B(x_{0},R)}\cup\overline{B(\hat{x},R')}}\{v(y)-\frac{1}{\epsilon}|x-y|^{2}\},\quad\forall x\in\overline{B(x_{0},R)}\cup\overline{B(\hat{x},R')},
\end{equation*}
respectively.

Then we conclude that there exists $\epsilon_{0}=\epsilon_{0}(\delta,\alpha,\E)$ such that for $0<\epsilon<\epsilon_{0}$,
\begin{equation}
\min\limits_{\overline{B(\hat{x},\delta)}}(u_{\epsilon}-v^{\epsilon}-\E h)\leq0,\quad u_{\epsilon}-v^{\epsilon}-\E h>0~\mbox{on }\overline{B(\hat{x},\delta)}\texttt{\symbol{'134}}B(\hat{x},\frac{1}{2}\delta).\label{319'}
\end{equation}
Indeed, the first part of (\ref{319'}) follows from the definitions of $u_{\epsilon}$ and $v^{\epsilon}$, and the fact that $h(\hat{x})=0$; $(u_{\epsilon}-v^{\epsilon}-\E h)(\hat{x})\leq (u-v)(\hat{x})=0$. Now we prove the second part of (\ref{319'}). By theorem 5.1 (a) in \cite{CabreCaffBook}, we have that 
\begin{equation*}
u_{\epsilon}-v^{\epsilon}\uparrow u-v\quad\mbox{uniformly on }B(x_{0},R)\cup B(\hat{x},R'),\quad\mbox{as }\epsilon\rightarrow0.
\end{equation*} 
It follows that for any $M>0$, there exists $\epsilon_{0}(M)>0$ such that
\begin{equation*}
(u_{\epsilon}-v^{\epsilon}-\E h)(x)>\min\limits_{\overline{B(\hat{x},\delta)}\texttt{\symbol{'134}}B(\hat{x},\frac{1}{2}\delta)}(u-v-\E h)-M
\end{equation*}
 for any $0<\epsilon<\epsilon_{0}$ and any $x\in\overline{B(\hat{x},\delta)}\texttt{\symbol{'134}}B(\hat{x},\frac{1}{2}\delta)$. Then by taking $0<M<\frac{1}{2}\min\limits_{\overline{B(\hat{x},\delta)}\texttt{\symbol{'134}}B(\hat{x},\frac{1}{2}\delta)}(u-v-\E h)$, (\ref{319'}) is obtained.

It follows from (\ref{319'}) that there exists $\bar{\eta}=\bar{\eta}(\delta,\alpha,\E)>0$ such that for any $\eta\in(0,\bar{\eta})$, there exists $\tau=\tau(\epsilon,\eta,\delta,\alpha,\E)\in\mathbb{R}^{1}$ such that 
\begin{equation}
\min\limits_{\overline{B(\hat{x},\delta)}}(u_{\epsilon}-v^{\epsilon}-\E h-\tau)=-\eta,\quad u_{\epsilon}-v^{\epsilon}-\E h-\tau>0~\mbox{on }\overline{B(\hat{x},\delta)}\texttt{\symbol{'134}}B(\hat{x},\frac{1}{2}\delta).
\end{equation}

Let $$\xi_{\epsilon}:=u_{\epsilon}-v^{\epsilon}-\E h-\tau,$$ and $\Gamma_{\xi_{\epsilon}^{-}}$ denote the convex envelope of $\xi_{\epsilon}^{-}:=-\min\{\xi_{\epsilon},0\}$ on $\overline{B(\hat{x},\delta)}$. Then by (20) in \cite{LiNgWang-arxiv} and (\ref{319-}), we have 
\begin{equation*}
\nabla^{2}\xi_{\epsilon}\leq\frac{4}{\epsilon}I+C\E I\quad\mbox{ a.e. in }B(\hat{x},\frac{1}{2}\delta).
\end{equation*}
And by lemma 3.5 in \cite{CabreCaffBook}, we have 
\begin{equation*}
\int_{\{\xi_{\epsilon}=\Gamma_{\xi_{\epsilon}^{-}}\}}\mbox{det}(\nabla^{2}\Gamma_{\xi_{\epsilon}^{-}})>0,
\end{equation*}
which implies that the Lebesgue measure of $\{\xi_{\epsilon}=\Gamma_{\xi_{\epsilon}^{-}}\}$ is positive. Then there exists $x_{\epsilon,\eta}\in\{\xi_{\epsilon}=\Gamma_{\xi_{\epsilon}^{-}}\}\cap B(\hat{x},\frac{1}{2}\delta)$ such that both of $v^{\epsilon}$ and $u_{\epsilon}$ are punctually second order differentiable at $x_{\epsilon,\eta}$, 
\begin{equation}
0>\xi_{\epsilon}(x_{\epsilon,\eta})\geq-\eta,\label{shitb}
\end{equation}
\begin{equation}
|\nabla\xi_{\epsilon}(x_{\epsilon,\eta})|\leq C\eta,\label{shitc}
\end{equation}
and
\begin{equation}
\nabla^{2}\xi_{\epsilon}(x_{\epsilon,\eta})=\nabla^{2}(u_{\epsilon}-v^{\epsilon}-\E h)(x_{\epsilon,\eta})\geq 0.\label{shit2}
\end{equation}

For $x_{\epsilon,\eta}\in\Omega$, by the definitions of $u_{\epsilon}$ and $v^{\epsilon}$, there exist $(x_{\epsilon,\eta})_{*}$ and $(x_{\epsilon,\eta})^{*}\in\Omega$ such that 
\begin{equation*}
u_{\epsilon}(x_{\epsilon,\eta})=u((x_{\epsilon,\eta})_{*})+\frac{1}{\epsilon}|(x_{\epsilon,\eta})_{*}-x_{\epsilon,\eta}|^{2},
\end{equation*}
and 
\begin{equation*}
v^{\epsilon}(x_{\epsilon,\eta})=v((x_{\epsilon,\eta})^{*})-\frac{1}{\epsilon}|(x_{\epsilon,\eta})^{*}-x_{\epsilon,\eta}|^{2}.
\end{equation*}

Since $u$ and $v\in C^{0,1}_{loc}(\Omega)$, by (2.6) and (2.7) in \cite{Li09-CPAM}, we have 
\begin{equation}
|(x_{\epsilon,\eta})_{*}-x_{\epsilon,\eta}|+|(x_{\epsilon,\eta})^{*}-x_{\epsilon,\eta}|\leq C_{1}\epsilon,\label{eq10}
\end{equation}
and 
\begin{equation}
|\nabla u_{\epsilon}(x_{\epsilon,\eta})|+|\nabla v^{\epsilon}(x_{\epsilon,\eta})|\leq C_{2},\label{eq11}
\end{equation}
where $C_{1}$ and $C_{2}$ are two universal positive constant independent of $\epsilon$ and $\eta$. 

Since $u_{\epsilon}$ is punctually second order differentiable at $x_{\epsilon,\eta}$, we have 
\begin{equation}
u_{\epsilon}(x_{\epsilon,\eta}+z)\geq u_{\epsilon}(x_{\epsilon,\eta})+\nabla u_{\epsilon}(x_{\epsilon,\eta})\cdot z+\frac{1}{2}z^{T}\nabla^{2}u_{\epsilon}(x_{\epsilon,\eta})z+o(|z|^{2}),\quad\mbox{ as }z\rightarrow 0.\label{yre2}
\end{equation}
By the definition of $u_{\epsilon}$, we have
\begin{equation*}
u_{\epsilon}(x_{\epsilon,\eta}+z)\leq u((x_{\epsilon,\eta})_{*}+z) +\frac{1}{\epsilon}|(x_{\epsilon,\eta})_{*}-x_{\epsilon,\eta}|^{2},\label{yree}
\end{equation*}
and therefore, in view of (\ref{yre2}),
\begin{align*}
u((x_{\epsilon,\eta})_{*}+z)&\geq u_{\epsilon}(x_{\epsilon,\eta}+z) - \frac{1}{\epsilon}|(x_{\epsilon,\eta})_{*}-x_{\epsilon,\eta}|^{2}\\
&\geq P_{\epsilon}((x_{\epsilon,\eta})_{*}+ z)+o(|z|^{2}),\quad\mbox{ as }z\rightarrow 0,
\end{align*}
where $P_{\epsilon}$ is a quadratic polynomial with 
\begin{align*}
P_{\epsilon}((x_{\epsilon,\eta})_{*})&=u_{\epsilon}(x_{\epsilon,\eta}) -\frac{1}{\epsilon}|(x_{\epsilon,\eta})_{*}-x_{\epsilon,\eta}|^{2} = u(x_{\epsilon,\eta})_{*}), \nonumber \\
\nabla P_{\epsilon}((x_{\epsilon,\eta})_{*})&=\nabla u_{\epsilon}(x_{\epsilon,\eta}),\\
\nabla^{2} P_{\epsilon}((x_{\epsilon,\eta})_{*})&=\nabla^{2} u_{\epsilon}(x_{\epsilon,\eta}).
\end{align*}
Since $u$ satisfies (\ref{eq0}) in the viscosity sense, we thus have 
\begin{equation}
\nabla^2 u_\epsilon(x_{\epsilon,\eta}) + L((x_{\epsilon,\eta})_{*}, \nabla u_{\epsilon}(x_{\epsilon,\eta})) = F[P_{\epsilon}]((x_{\epsilon,\eta})_{*})\in \Snn\setminus U.\label{fanyang}
\end{equation}
On the other hand, in view of (\ref{eq10}), (\ref{eq11}) and the fact that $L\in C_{loc}^{0,1}(\Omega\times\RR^{n})$, 
\begin{equation}
L(x_{\epsilon,\eta}, \nabla u_{\epsilon}(x_{\epsilon,\eta})) - L((x_{\epsilon,\eta})_{*}, \nabla u_{\epsilon}(x_{\epsilon,\eta})) \leq C|x_{\epsilon,\eta} - (x_{\epsilon,\eta})_{*}|I\leq a_{1}\epsilon I,\label{fanyang--'}
\end{equation}
where $C$ and $a_{1}>0$ are universal constants.

It follows from (\ref{23}), (\ref{fanyang}) and (\ref{fanyang--'}) that 
\begin{equation}
F[u_{\epsilon}](x_{\epsilon,\eta})-a_{1}\epsilon I\in \mathcal{S}^{n\times n}\texttt{\symbol{'134}}U.\label{dd-}
\end{equation}
Analogusly, we can obtain 
\begin{equation*}
F[v^{\epsilon}](x_{\epsilon,\eta})+a_{2}\epsilon I\in \bar{U}
\end{equation*}
for some universal constants $a_{2}>0$.

By (\ref{shitc}), (\ref{shit2}), (\ref{eq11}) and the fact that $L\in C_{loc}^{0,1}(\Omega\times\RR^{n})$,
\begin{align}
F[u_{\epsilon}](x_{\epsilon,\eta})&\geq \nabla^{2}(v^{\epsilon}+\E h)(x_{\epsilon,\eta})+L(x_{\epsilon,\eta},\nabla u_{\epsilon}(x_{\epsilon,\eta}))\nonumber\\
&=F[v^{\epsilon}+\E h](x_{\epsilon,\eta})+L(x_{\epsilon,\eta},\nabla u_{\epsilon}(x_{\epsilon,\eta}))-L(x_{\epsilon,\eta},\nabla v^{\epsilon}(x_{\epsilon,\eta}))\nonumber\\
&\geq F[v^{\epsilon}+\E h](x_{\epsilon,\eta})-C|\nabla(u_{\epsilon}-v^{\epsilon})(x_{\epsilon,\eta})|\nonumber\\
&\geq F[v^{\epsilon}+\E h](x_{\epsilon,\eta})-C(\eta+\hat{\varepsilon}|\nabla h(x_{\epsilon,\eta})|)I.\label{a1}
\end{align}

By (\ref{319-}), (\ref{eq11}) and the fact that $L\in C_{loc}^{0,1}(\Omega\times\RR^{n})$, we have 
\begin{align}
&\quad F[v^{\epsilon}+\E h](x_{\epsilon,\eta})\nonumber\\
&=F[v^{\epsilon}](x_{\epsilon,\eta})+\E\nabla^{2}h(x_{\epsilon,\eta})+L(x_{\epsilon,\eta},\nabla(v^{\epsilon}+\E h)(x_{\epsilon,\eta}))-L(x_{\epsilon,\eta},\nabla v^{\epsilon}(x_{\epsilon,\eta}))\nonumber\\
&\geq F[v^{\epsilon}](x_{\epsilon,\eta})+\E[\nabla^{2}h(x_{\epsilon,\eta})-C |\nabla h(x_{\epsilon,\eta})|I].\label{9801}
\end{align}

Then by (\ref{hi}), (\ref{hij}) and the fact $|x_{\epsilon,\eta}-x_{0}|<2R$,
\begin{align}
&\quad\nabla^{2}h(x_{\epsilon,\eta})-C |\nabla h(x_{\epsilon,\eta})|I\nonumber\\
&=4\alpha^{2}e^{-\alpha|x_{\epsilon,\eta}-x_{0}|^{2}}\left[(x_{\epsilon,\eta}-x_{0})\otimes(x_{\epsilon,\eta}-x_{0})-\frac{1}{2\alpha}I-\frac{C}{4\alpha}|x_{\epsilon,\eta}-x_{0}|I\right]\nonumber\\
&\geq 4\alpha^{2}e^{-\alpha|x_{\epsilon,\eta}-x_{0}|^{2}}\left[(x_{\epsilon,\eta}-x_{0})\otimes(x_{\epsilon,\eta}-x_{0})-\frac{C}{\alpha}I\right]\nonumber\\
&\geq 4\alpha^{2}e^{-\alpha|x_{\epsilon,\eta}-x_{0}|^{2}}[(\hat{x}-x_{0})\otimes(\hat{x}-x_{0})-C\delta RI-\frac{C}{\alpha}I]\nonumber\\
&=4R^{2}\alpha^{2}e^{-\alpha|x_{\epsilon,\eta}-x_{0}|^{2}}\left[\left(\frac{\hat{x}-x_{0}}{R})\otimes(\frac{\hat{x}-x_{0}}{R}\right)-C\delta I-\frac{C}{\alpha}I\right]\nonumber\\
&\geq 4R^{2}\alpha^{2}e^{-4R^{2}\alpha}\left[\left(\frac{\hat{x}-x_{0}}{R})\otimes(\frac{\hat{x}-x_{0}}{R}\right)-C\delta I-\frac{C}{\alpha}I\right].\label{980}
\end{align}

Inserting (\ref{980}) into (\ref{9801}), we have
\begin{equation}
F[v^{\epsilon}+\varepsilon h](x_{\epsilon,\eta})\geq F[v^{\epsilon}](x_{\epsilon,\eta})+4R^{2}\E\alpha^{2}e^{-4R^{2}\alpha}\left[\left(\frac{\hat{x}-x_{0}}{R})\otimes(\frac{\hat{x}-x_{0}}{R}\right)-C\delta I-\frac{C}{\alpha}I\right].\label{a2}
\end{equation}

It follows from (\ref{a1}) and (\ref{a2}) that 
\begin{align}
&\quad F[u_{\epsilon}](x_{\epsilon,\eta})-a_{1}\epsilon I\nonumber\\
&\geq F[v^{\epsilon}](x_{\epsilon,\eta})+a_{2}\epsilon I\nonumber\\
&\quad\quad\quad\quad+4R^{2}\E\alpha^{2}e^{-4R^{2}\alpha}\left[\left(\frac{\hat{x}-x_{0}}{R})\otimes(\frac{\hat{x}-x_{0}}{R}\right)-C\delta I-\frac{C}{\alpha}I-C\frac{e^{4R^{2}\alpha}}{\E\alpha^{2}}(\epsilon+\eta)I\right].\label{dd'}
\end{align}

We can firstly fix the value of small $\delta>0$ and a large $\alpha>1$, then fix the value of small $\E>0$, and lastly fix the value of small $\epsilon$ and $\eta>0$ such that 
\begin{equation*}
\|C\delta I+\frac{C}{\alpha}I+C\frac{e^{4R^{2}\alpha}}{\E\alpha^{2}}(\epsilon+\eta)I\|<\frac{1}{2}\mu(\frac{\hat{x}-x_{0}}{R}),
\end{equation*}
where $\mu$ is obtained from condition (\ref{halfcone}). 

Therefore, by (\ref{halfcone}) and (\ref{dd'}), we have that 

\begin{equation*}
F[u_{\epsilon}](x_{\epsilon,\eta})-a_{1}\epsilon I\in U,
\end{equation*}
which is a contradiction with (\ref{dd-}). Theorem \ref{strong compare} is proved.

\end{proof}

\section{Proof of Theorem \ref{hopf}}

\begin{proof}[Proof of Theorem \ref{hopf}]

Since $\partial\Omega$ is $C^{2}$ near $\hat{x}$, there exists an open ball $B(x_{0},R)\subset\Omega$ such that $\overline{B(x_{0},R)}\cap\partial\Omega=\{\hat{x}\}$ and 
\begin{equation*}
\begin{cases}
u-v>0,\quad\mbox{in }\overline{B(x_{0},R)}\texttt{\symbol{'134}}\{\hat{x}\},\\
u(\hat{x})-v(\hat{x})=0.
\end{cases}
\end{equation*}

Let $h$ be defined as in (\ref{comparef}). We work in the domain
$$A_{\delta}:=B(\hat{x},\delta)\cap B(x_{0},R).$$
It is easy to see that 
\begin{equation*}
u-v\geq\E h,\quad\mbox{on }\partial{A_{\delta}}
\end{equation*}
for any $0<\delta<\frac{R}{2}$ and $0<\E<\min\limits_{\partial B(\hat{x},\delta)\cap \overline{B(x_{0},R)}}(u-v)$.

We claim that for $\varepsilon$ small enough, 
\begin{equation*}
u-v\geq\E h,\quad\mbox{on }\overline{A_{\delta}}.
\end{equation*}

Once the claim is proved, then we have that 
\begin{equation*}
\liminf\limits_{s\rightarrow0^{+}}\frac{(u-v)(\hat{x}+s\nu(\hat{x}))}{s}\geq\E\liminf\limits_{s\rightarrow0^{+}}\frac{h(\hat{x}+s\nu(\hat{x}))}{s}=2\alpha Re^{-\alpha R^{2}}>0.
\end{equation*}
Therefore, in order to finish the proof of Theorem \ref{hopf}, we only need to prove the above claim. Suppose the contrary, that is, 
\begin{equation*}
\zeta=\zeta(\E,\alpha,\delta):=\min\limits_{\overline{A_{\delta}}}(u-v-\E h)<0.
\end{equation*}
It follows that 
\begin{equation*}
\min\limits_{\overline{A_{\delta}}}(u-v-\E h-\zeta)=0,\quad u-v-\E h-\zeta\geq-\zeta>0~\mbox{on }\partial{A_{\delta}}.
\end{equation*}

Now we can follow the argument as in the proof of Theorem \ref{strong compare} to get a contradiction.  Theorem \ref{hopf} is proved.

\end{proof}

\noindent{\bf{\large Acknowledgments.}} Li is partially supported by NSF grant DMS-1501004. Wang is partially supported by NNSF (11701027). 


\newcommand{\noopsort}[1]{}

\end{document}